\documentclass[12pt]{amsart}
\usepackage{fullpage}
\usepackage[T1]{fontenc}
\usepackage{amsmath,amssymb,amsthm,mathtools,amscd}
\usepackage{booktabs,array}
\usepackage{tikz-cd}
\usepackage[colorlinks=true,linkcolor=red,urlcolor=blue,citecolor=blue]{hyperref}
\hypersetup{pdftitle={Compatibility of Perverse Truncation of Ngô Fibrations}}

\makeatletter
\def\thm@space@setup{%
  \thm@preskip=2ex \thm@postskip=2ex
}
\makeatother
\usepackage{todonotes}
\definecolor{pistachio}{rgb}{0.58, 0.77, 0.45}
\definecolor{eggshell}{rgb}{0.94, 0.92, 0.84}

\usepackage{mathrsfs}
\newtheorem{theorem}{Theorem~}[section]
\newtheorem{proposition}[theorem]{Proposition~}
\newtheorem{lemma}[theorem]{Lemma~}
\newtheorem{corollary}[theorem]{Corollary~}

\newtheorem{assumption}[theorem]{Assumption}
\theoremstyle{definition}
\newtheorem{definition}[theorem]{Definition~}
\theoremstyle{definition}
\newtheorem{remark}[theorem]{Remark~}

\numberwithin{equation}{section}
\newcommand{\Q}{\mathbb Q}

\newcommand{\C}{\mathbb C}
\newcommand{\Z}{\mathbb Z}
\newcommand{\Db}{D^b_c}
\newcommand{\ptau}{{}^p\!\tau}
\newcommand{\pH}{{}^p\!\mathcal H}
\newcommand{\IC}{\operatorname{IC}}
\newcommand{\Hom}{\operatorname{Hom}}
\newcommand{\Ext}{\operatorname{Ext}}

\newcommand{\im}{\operatorname{im}}
\newcommand{\id}{\operatorname{id}}
\newcommand{\Lie}{\operatorname{Lie}}
\newcommand{\Ab}{\operatorname{Ab}}
\newcommand{\Pic}{\operatorname{Pic}}

\newcommand{\GL}{\operatorname{GL}}
\newcommand{\SL}{\operatorname{SL}}

\newcommand{\ad}{\operatorname{ad}}
\newcommand{\rk}{\operatorname{rk}}

\newcommand{\RHom}{R\!\operatorname{Hom}}
\newcommand{\cHom}{\mathcal{H}om}
\newcommand{\cExt}{\mathcal{E}xt}
\newcommand{\cRHom}{R\cHom}
\title[Compatibility of perverse truncation]{Compatibility of Perverse Truncation of \\ Ng\^o Fibrations}
\author{Bin Wang}
\address{Hetao Institute of Mathematics and Interdisciplinary Sciences, Shenzhen, China}
\email{matph92@gmail.com}
\author{Shihao Wang}
\address{Qiuzhen College, Tsinghua University, Beijing, China}
\email{wangshihao25@mails.tsinghua.edu.cn}
\author{Chenglong Yu}
\address{Center for Mathematics and Interdisciplinary Sciences, Fudan University and Shanghai Institute for Mathematics and Interdisciplinary Sciences (SIMIS), Shanghai, China}
\email{yuchenglong@simis.cn}
\date{}
\begin{document}
\begin{abstract}
We study compatibility of perverse truncations for morphisms between Ng\^o fibrations with rationally smooth total spaces. We show that, under an effective rank inequality, equivariant pullbacks preserve the normalized perverse truncations. Applied to the diagonal, this gives multiplicativity of perverse truncations for Hitchin systems with rationally smooth total spaces.
\end{abstract}
\maketitle
\tableofcontents

\section{Introduction}
Ng\^o introduced weak abelian fibrations and established his support theorem in the proof of the fundamental lemma \cite[Section~7]{Ngo}. Commutative group actions and $\delta$-regularity constrain the supports and local systems occurring in the decomposition theorem. De Cataldo--Rapagnetta--Sacc\`a refined this description in terms of Ng\^o strings \cite[Appendix~A]{DCRS}, while Maulik--Shen extended the support theorem to suitable equivariant complexes, including intersection complexes \cite[Theorems~1.1 and~1.8]{MSChi}. More recently, de Cataldo--Fringuelli--Fernandez Herrero--Mauri introduced the notion of a Ng\^o fibration and proved that canonical and logarithmic Hitchin fibrations with positive twisting degree are Ng\^o fibrations \cite[Definition~1.1 and Theorem~A]{DCFFM}. We follow the notion of Ng\^o fibration in \cite[Definition~1.1]{DCFFM}: a weak abelian fibration satisfying $\delta$-regularity and a cohomological bound on the direct image of its intersection complex. For rationally smooth total spaces considered here, the cohomological bound follows from the fiber-dimension bound.

This paper studies compatibility of perverse truncations for Ng\^o fibrations with rationally smooth total spaces. We give a criterion for the natural pullback associated with an equivariant morphism of fibrations to preserve the perverse truncations. In fact, we study more general fibrations  and the criterion is motivated by the $\delta$-regularity and weak abelian fibrations.

Throughout, varieties and reductive groups are over $\C$. The basic object of study is a \emph{comparison} $(F,a)$, by which we mean a commutative square
\begin{equation}\label{eq:intro-comparison}
\begin{tikzcd}
X\ar[r,"F"]\ar[d,"h_X"']&Y\ar[d,"h_Y"]\\
B_X\ar[r,"a"']&B_Y.
\end{tikzcd}
\end{equation}
The square need not be Cartesian. Set
\[
K_X=Rh_{X*}\Q_X,\qquad K_Y=Rh_{Y*}\Q_Y,
\qquad d_X=\dim B_X,\quad d_Y=\dim B_Y.
\]
Then we have the natural map:
\[
\beta_{F,a}:a^*K_Y\longrightarrow K_X.
\]
On both sides, we have the perverse truncations:
\[
\ptau_{\le d_X+k}K_X, \qquad \ptau_{\le d_Y+k}K_Y
\]
Following conventions in \cite{MSYCycles}, the indices are normalized by base dimension, so that for a proper fibration the index $k$ agrees with fiber cohomological degree. 
We ask whether $\beta_{F,a}$ preserves the perverse truncations. More precisely, whether $\beta_{F,a}$ admits compatible factorizations
\begin{equation}\label{eq:intro-weak}
a^*(\ptau_{\le d_Y+k}K_Y)\longrightarrow
\ptau_{\le d_X+k}K_X,\qquad k\in\Z.
\end{equation}
These factorizations imply
\[
F^*P_kH^m(Y,\Q)\subseteq P_kH^m(X,\Q).
\]
We call \eqref{eq:intro-weak} \emph{the compatibility of perverse truncation} with respect to the comparison. 

\subsection{Compatibility under equivariant comparison}
Let $(X,P_X,B_X)$ and $(Y,P_Y,B_Y)$ be projective weak abelian fibrations, as in \cite[Section~1.1]{MSChi}, see also Section~\ref{subsec:weak-abelian}. Assume that the bases are smooth, the commutative group schemes are smooth separated with connected fibers, and, for $Z=X,Y$,
\[
\dim Z=2d_Z,\qquad
\text{every nonempty fiber of }h_Z\text{ is pure of dimension }d_Z.
\]
Assume also that $X$ and $Y$ are rationally smooth, in particular their intersection complexes are shifted constant sheaves.

Suppose the comparison $(F,a)$ is equivariant for a smooth separated finite-type commutative group algebraic space $Q/B_X$ with connected fibers and homomorphisms
\[
P_X\xleftarrow{\,u\,}Q\xrightarrow{\,v\,}a^*P_Y.
\]
Explicitly, $F_a:X\to Y\times_{B_Y}B_X$, $x\mapsto(F(x),h_X(x))$, satisfies
$F_a(u(q)\cdot x)=v(q)\cdot F_a(x)$.
Write $\Ab(P_b)$ for the abelian quotient of a connected commutative algebraic group $P_b$, and set
\begin{equation}\label{eq:intro-rank}
r(b)=\dim\im\bigl(\Ab(Q_b)\longrightarrow\Ab(P_{X,b})\bigr),
\qquad \dim\{b\in B_X:r(b)\le c\}\le c\quad(c\ge0).
\end{equation}
The bound implies the $\delta$-regularity of $(X, P_X, B_X)$. If the map on abelian quotients is surjective, the bound is precisely $\delta$-regularity of $P_X/B_X$.

\begin{theorem}\label{thm:intro-comparison}\label{thm:intro-weak}
Under these hypotheses, the comparison $(F,a)$ is compatible with perverse truncation in the sense of Definition~\ref{def:weak}. In particular, $F^*$ preserves the perverse filtration.
\end{theorem}

This is the weak-abelian-fibration case of Theorem~\ref{thm:common-action-weak}. If the rank bound is relaxed to $\dim\{r\le c\}\le c+\kappa$, the conclusion has target $\ptau_{\le d_X+k+\kappa}K_X$.
\begin{remark}
Theorem~\ref{thm:common-action-weak} requires $(X,P_X,B_X)$ to be a weak abelian fibration. The right hand side needs a smooth separated commutative group action, but neither polarizability nor affine  stabilizers. The weak-abelian-fibration formulation above is a convenient special case.
\end{remark}
\subsection{Application to Hitchin fibrations: multiplicativity}
Hitchin fibrations provide the main applications. Their perverse filtrations enter the $P=W$ conjecture of de Cataldo--Hausel--Migliorini \cite{DCHMPW}, which, under non-abelian Hodge theory, identifies the perverse filtration $P$ on the Dolbeault moduli space $M_{\mathrm{Dol}}$ with the weight filtration $W$ on the corresponding twisted character variety $M_{\mathrm B}$:
\[
 P_kH^m(M_{\mathrm{Dol}},\Q)
 =W_{2k}H^m(M_{\mathrm B},\Q)
 =W_{2k+1}H^m(M_{\mathrm B},\Q).
\]
For curves of genus at least two, the coprime $\GL_n$ conjecture was proved by Maulik--Shen \cite{MSPW} and Hausel--Mellit--Minets--Schiffmann \cite{HMMSSPW}; Maulik--Shen--Yin gave a Fourier-theoretic proof \cite[Theorem~0.4]{MSYFourier}. Maulik--Shen--Yin also established motivic multiplicativity for dualizable abelian fibrations satisfying Fourier vanishing, whose realization gives multiplicativity of perverse truncations \cite[Theorem~2.6(ii)--(iii)]{MSYFourier}; see also \cite[Theorem~3.2(b)]{MSYDualizable}. 

The multiplicativity of perverse truncations plays an important role in the proof of the $P=W$ conjecture. To see how the compatibility can be used to prove multiplicativity: the diagonal morphism is equivariant, and its pullback is cup-product; hence multiplicativity follows from the compatibility.

Let $(X,P,B)$ be a Ng\^o fibration with $d=\dim B$. The comparison
\[
(\Delta_X,\Delta_B):(X\to B)\longrightarrow(X\times X\to B\times B)
\]
is equivariant for $P$ acting diagonally on the product. The K\"unneth formula identifies the ``$\beta_{F,a}$'' with the cup product.

\begin{theorem}\label{thm:intro-multiplicativity}
For $K=Rh_*\Q_X$ and $r,s\in\Z$, cup-product factors as
\begin{equation}\label{eq:intro-mult}
\ptau_{\le r}K\otimes^{\mathbf L}\ptau_{\le s}K
\longrightarrow\ptau_{\le r+s-d}K\longrightarrow K.
\end{equation}
Consequently $P_iH^m(X,\Q)\cup P_jH^n(X,\Q)\subseteq P_{i+j}H^{m+n}(X,\Q)$.
This includes the classical $\GL_N$ Higgs moduli spaces, the twisted
traceless $\SL_N$ spaces $M_{\SL_N}^{\Lambda,0}(C,\varnothing)$,
and the full-flag strongly parabolic spaces, with the specified determinant and weights
where applicable, under
Assumption~\ref{assum:parabolic-application}.
\end{theorem}
The shift is determined by the comparison of bases: $\dim(B\times B)=2d$, whereas $\dim B=d$. Substituting $r=d+i$ and $s=d+j$ gives the normalized bound $d+i+j$.

\subsection{Proof Strategy}
The argument uses the weak abelian fibration structure, rational smoothness and the effective-rank inequality \eqref{eq:intro-rank}. For a Ng\^o fibration, the last condition follows from $\delta$-regularity when the common group surjects onto the abelian quotients of the source symmetry groups. For the application to Hitchin systems, the geometric hypotheses are collected in Assumption~\ref{assum:parabolic-application}.

The proof first uses the group action to make $\beta_{F,a}$ a morphism of modules of the homological algebras. The abelian part of the action on $X$ then supplies a free exterior-algebra factor. Combined with stalk and costalk bounds, this gives a global Ext vanishing range controlled by \eqref{eq:intro-rank} which removes the obstruction to the compatible factorization \eqref{eq:intro-weak}.

Section~\ref{sec:background} introduces comparison morphisms, perverse truncations and weak abelian fibrations. Section~\ref{weak:section} proves the compatibility criterion and its extensions. Section~\ref{sec:multiplicativity} gives the applications: deriving multiplicativity from the diagonal. 
\subsection{Acknowledgements}
The authors thank Tao Su, Dingxin Zhang, Zili Zhang, and Zhiwei Zheng for their interest in this work and for helpful discussions. 

The authors had been considering the compatibility of perverse truncations for morphisms between Ng\^o fibrations \cite{Ngo} before AI was involved in this work. The authors initially established the argument for the local system part of each intersection complex. AI assisted with developing the technical arguments for the boundary contributions, writing and revising the manuscript, and reviewing proofs. The authors take full responsibility for the content of the paper.
\section{Preliminaries}\label{sec:background}\label{sec:comparisons}
Cohomology has rational coefficients if not specified. We write $\Db(B,\Q)$ for the bounded constructible derived category, and $\ptau$ and $\pH$ for middle-perversity truncation and cohomology. The comparison morphisms and compatibility with perverse truncations are initially considered in the ordinary constructible derived categories.
The enhanced categories used in the proof are specified in Subsection~\ref{subsec:categorical-setup}.

\subsection{Comparisons and their induced morphisms}\label{subsec:comparison-data}
\begin{definition}\label{def:comparison}
A \emph{comparison} from $h_X:X\to B_X$ to $h_Y:Y\to B_Y$ is a pair of algebraic morphisms
\[
F:X\to Y,\qquad a:B_X\to B_Y,
\qquad h_Y\circ F=a\circ h_X.
\]
\end{definition}
We assume the morphisms $h_X,h_Y$ are projective and the bases are smooth. 
Set $Y_a=Y\times_{B_Y}B_X$, let $h_{Y,a}:Y_a\to B_X$ be its projection, and define $F_a:X\to Y_a$ by $F_a(x)=(F(x),h_X(x))$.
We have the following diagram:
\[
\begin{tikzcd}
    X\ar[r,"F_a"]\ar[rd,"h_X"]&Y_a\ar[d,"h_{Y_a}"]\ar[r]&Y\ar[d,"h_Y"]\\
    &B_X\ar[r,"a"]&B_Y
\end{tikzcd}
\]
where the right square is Cartesian. Proper base change gives the natural map
\begin{equation}\label{eq:actual-beta}
\beta_{F,a}:a^*K_Y
\xrightarrow{\ \sim\ }Rh_{Y,a*}\Q_{Y_a}
\longrightarrow Rh_{Y,a*}RF_{a*}\Q_X
=K_X .
\end{equation}
In other words, $\beta$ is the pushforward of the natural map $\Q_{Y_a}\rightarrow RF_{a*}\Q_X$ by $Rh_{Y_a*}$.

\subsection{Perverse truncations and compatibility}\label{subsec:normalized-comparison}
Put $d_X=\dim B_X$ and $d_Y=\dim B_Y$. We use the base-dimension normalization customary for the Hitchin perverse filtration, as in \cite[Section~1.1, (2)--(4)]{MSYCycles}, and define
\begin{equation}\label{eq:cohom-P}
P_kH^m(X,\Q)=\im\bigl(
H^m(B_X,\ptau_{\le d_X+k}K_X)\longrightarrow H^m(X,\Q)\bigr).
\end{equation}
For a smooth proper map, $\pH^{d_X+q}(K_X)=(R^qh_{X*}\Q_X)[d_X]$, which explains the normalization. Equivalently, if $\widehat K_X=K_X[d_X]$, then
\[
P_kH^m(X,\Q)=\im\bigl(H^{m-d_X}(B_X,\ptau_{\le k}\widehat K_X)\longrightarrow H^m(X,\Q)\bigr).
\]
Maulik--Shen--Yin\cite{MSYCycles} shift their direct image by $\dim X-d_h$, where $d_h$ is the relative dimension. In our dimension range $\dim X=2d_X$ and $d_h=d_X$, this is exactly $\widehat K_X$. In particular the absolute truncation bound in their Conjecture~6.1 is $r+s-(\dim X-d_h)=r+s-d_X$, as in \eqref{eq:intro-mult}. A convention using $Rh_{X*}\IC_X=K_X[2d_X]$ instead has the centered perverse index $k-d_X$; it is not the normalization used here. Arbitrary inverse image does not commute with perverse truncation, so it is important to retain the order of operations in the following definition.

\begin{definition}\label{def:weak}
The comparison $(F,a)$ is \emph{compatible with perverse truncation} if for any $k\in\Z$, the composite has the factorization
\begin{equation}\label{eq:compatibility}
\begin{tikzcd}
    a^*\ptau_{\le d_Y+k}K_Y\ar[r]\ar[rd, dotted]&a^*K_Y\ar[r,"\beta_{F,a}"]&K_X\\
    &\ptau_{\le d_X+k}K_X\ar[ru]
\end{tikzcd}
\end{equation}
The comparison morphisms and compatibility with perverse truncations are formulated in the ordinary bounded
constructible derived category.
\end{definition}
\begin{remark}
    If $h_X$ is projective and $X$ is a rational homology manifold, the decomposition theorem implies that each canonical morphism $\ptau_{\le m}K_X\to K_X$ is a split monomorphism in the derived category. Consequently, a factorization through this morphism, when it exists, is unique. The factorizations in \eqref{eq:compatibility} therefore commute automatically with the truncation transition maps.
\end{remark}
Applying hypercohomology to \eqref{eq:compatibility}, immediately gives
\[
F^*P_kH^m(Y,\Q)\subseteq P_kH^m(X,\Q).
\]
For comparisons $(F,a):(X\to B_X)\to(Y\to B_Y)$ and
$(G,b):(Y\to B_Y)\to(Z\to B_Z)$, we have
\[
\beta_{G\circ F,b\circ a}
=\beta_{F,a}\circ a^*\beta_{G,b}.
\]
If both comparisons are compatible, their compatible factorizations compose as
\begin{equation}\label{eq:comparison-composition}
(b\circ a)^*\ptau_{\le d_Z+k}K_Z
=a^*b^*\ptau_{\le d_Z+k}K_Z\longrightarrow a^*\ptau_{\le d_Y+k}K_Y
\longrightarrow\ptau_{\le d_X+k}K_X.
\end{equation}
Thus compatibility of perverse truncation is preserved by composition, with no additional normalized shift.

\subsection{Weak abelian fibrations}\label{subsec:weak-abelian}In general, the compatibility of perverse truncation cannot be expected. We here introduce the \emph{weak abelian fibrations} which play an important role in the study of Hitchin systems.  We follow \cite[Section~1.1]{MSChi}. Let $h:M\to B$ be proper, with $M$ quasiprojective, and let $g:P\to B$ be a smooth commutative group scheme with geometrically connected fibers, acting on $M$ over $B$. All the group schemes used here are of finite type.

\begin{definition}\label{def:weak-abelian}
The triple $(M,P,B)$ is a \emph{weak abelian fibration of relative dimension $s$} if
\begin{enumerate}
\item every fiber of $P\to B$ has dimension $s$, and $M$ is pure-dimensional with $\dim M=s+\dim B$;
\item the action of $P$ on $M$ has affine stabilizers;
\item the relative Tate module of $P/B$ is polarizable.
\end{enumerate}
\end{definition}
The definition in \cite[Definition~4.36]{DCFFM} permits disconnected group fibers and imposes polarizability on the neutral component. Passing to that component gives the connected-fiber convention used here.

More precisely, the Tate module can be written
\[
T_{\overline{\Q}_\ell}(P)
=\mathcal H^{-1}\bigl(Rg_!\overline{\Q}_\ell[2s](s)\bigr).
\]
If we consider the rational coefficients, its stalk is $H_1(P_b,\Q)$. Following \cite[Section~1.3]{MSChi}, polarizability means that an \emph{\'etale-local} bilinear pairing descends, on each fiber, to a nondegenerate pairing on the Tate module of the abelian quotient. Equivalently, its fiberwise radical is the Tate module of the connected affine part. 

For a point $b$, by Chevalley's theorem,
\[
1\longrightarrow P_b^{\mathrm{aff}}\longrightarrow P_b
\longrightarrow\Ab(P_b)\longrightarrow1,
\qquad \delta(b)=\dim P_b^{\mathrm{aff}}.
\]
The group scheme is \emph{$\delta$-regular} if
\begin{equation}\label{eq:bg-delta}
\operatorname{codim}_B\{b:\delta(b)\ge m\}\ge m
\qquad(m\ge0).
\end{equation}
When $s=\dim B=d$, the identity $\dim\Ab(P_b)=d-\delta(b)$ rewrites this as
\[
\dim\{b:\dim\Ab(P_b)\le c\}\le c\qquad(c\ge0).
\]
Definition~\ref{def:weak-abelian} fixes $\dim M=s+\dim B$ and already requires $h$ to be proper. The applications impose the further conditions $s=\dim B$, smoothness of $B$, rational smoothness of $M$, projectivity of $h$, pure fiber dimension $s$, and $\delta$-regularity. These extra conditions are stated separately whenever they are needed.

A pure-dimensional variety $M$ is rationally smooth, or a rational homology manifold, when $\Q_M[\dim M]\cong\IC_M$. This allows us to apply the decomposition theorem to the constant direct images carrying the natural pullback.
\section{Compatibility of perverse truncation}\label{weak:section}
\label{subsec:rhm-criterion}
Motivated by \emph{weak abelian fibrations}, we consider comparisons with a common action by a commutative group space.
\begin{definition}
Let $(F,a):(X\to B_X)\to(Y\to B_Y)$ be a comparison. We say it is an \emph{equivariant comparison}, if additionally, there is a smooth separated finite-type commutative group algebraic space $Q/B_X$ with connected fibers and homomorphisms $u:Q\to P_X$ and $v:Q\to a^*P_Y$ such that:
\[
F_a\bigl(u(q)\cdot x\bigr)=v(q)\cdot F_a(x).
\]
\end{definition}
Throughout this section, we make the following geometric assumptions on the comparisons

\begin{assumption}\label{assum:geometry of comparisons}
    For notational convenience, we denote  $Z=X,Y$. 
\begin{enumerate}
    \item $Z$ is a pure-dimensional rational homology manifold (RHM).
    \item $d_Z=\dim B_Z, \dim Z=2d_Z, \dim h_Z^{-1}(b)=d_Z$, for every nonempty fiber, which is required to be pure-dimensional.
    \item We have $P_Z/B_Z$, smooth separated finite-type commutative group scheme with connected fibers, acting on $Z$.
    \item $(X,P_X, B_X)$ is a weak abelian fibration, i.e., the relative Tate module $T_{\bar{\mathbb Q}_\ell}(P_X)$ is \'etale-locally polarizable and the $P_X$-stabilizers are affine.
\end{enumerate}
\end{assumption}
We then need the following numerical assumption on \emph{the equivariant comparison}.
\begin{assumption}\label{assum:effective-rank inequality}
Put
\[
 r(b)=\dim\im(\Ab Q_b\longrightarrow\Ab P_{X,b}).
\]
We call the following:
\[
\dim\{b\in B_X:r(b)\le c\}\le c\quad(c\ge0)
\]
effective-rank inequality.
\end{assumption}
\begin{remark}
Suppose $P_X/B_X$ has relative dimension $d_X$, and put $A(b)=\dim\Ab(P_{X,b})$. Since $r(b)\le A(b)$, the effective-rank inequality implies
\[
\dim\{b:A(b)\le c\}\le\dim\{b:r(b)\le c\}\le c,
\]
which is precisely $\delta$-regularity of $P_X/B_X$. 
\end{remark}

\begin{theorem}\label{thm:common-action-weak}
Under these hypotheses, the comparison $(F,a)$ is compatible with perverse truncation. More precisely, for every $k\in\mathbb Z$, the natural composite
\[
a^*(\ptau_{\le d_Y+k}K_Y)
\longrightarrow a^*K_Y
\xrightarrow{\beta_{F,a}}K_X
\]
factors as
\begin{equation}\label{eq:rhm-weak-towers}
a^*(\ptau_{\le d_Y+k}K_Y)
\longrightarrow\ptau_{\le d_X+k}K_X
\longrightarrow K_X.
\end{equation}
Consequently, $F^*P_kH^m(Y,\mathbb Q)\subseteq P_kH^m(X,\mathbb Q)$.
\end{theorem}

The proof is organized into four parts. The geometric input is isolated in Proposition~\ref{prop:weak-costalk-free}; the numerical calculation is contained in Lemma~\ref{lem:weak-degree-bound} and Proposition~\ref{prop:weak-global-ext}. The final two lemmas improve the filtered degree and descend the resulting truncation factorizations to rational coefficients. The proof of Theorem~\ref{thm:common-action-weak} appears after these results.

\subsection{Notation and the initial filtered comparison}
\label{subsec:categorical-setup}

Throughout the proof we retain the geometric hypotheses of Theorem~\ref{thm:common-action-weak}. The intermediate constructions have coefficients in $\mathbb Q_\ell$; compatibility over $\mathbb Q$ is obtained in Lemma~\ref{lem:weak-rational-descent}.

For each base $B$, let $\mathscr D(B)$ denote the ambient presentable symmetric monoidal stable $\infty$-category
in the pro-\'etale realization of sheaf complexes used below. Its tensor product preserves colimits separately in each
variable. We use the enhanced six operations of \cite{LZ} and the realization of bounded constructible complexes in
\cite[Lemma~6.5.5 and Proposition~6.8.14]{BS}. Writing $\mathscr D_c^b(B)$ for the bounded constructible subcategory, we have
\[
h\mathscr D_c^b(B)\simeq\Db(B,\mathbb Q_\ell),
\]
where $h$ denotes the homotopy category.

Algebras are unital associative algebra objects in $\mathscr D(B)$, and module structures and module morphisms
are understood in this enhanced setting. For an algebra object $\mathcal A$, we write
\[
D(\mathcal A)
=
h\operatorname{Mod}_{\mathcal A}(\mathscr D(B)).
\]
The ambient categories accommodate module constructions, direct sums, and Rees objects; bounded constructibility
is imposed on the geometric complexes under consideration.

Compatibility in Definition~\ref{def:weak} is a property in the ordinary constructible derived category.
The enhanced module structures introduced here are used to prove that property.

For a smooth morphism $f$ of relative dimension $q$, denote its smooth trace map by
\[
 \operatorname{Tr}_f:
 Rf_!\mathbb{Q}_\ell[2q](q)\longrightarrow\mathbb{Q}_\ell.
\]
Let $\pi:Q\to B_X$ have relative dimension $q$, and set $\mathcal A=R\pi_!\mathbb{Q}_\ell[2q](q)$. Consider the following Cartesian diagram:
\[
\begin{tikzcd}
    Q\times_{B_X}Q\ar[r,"p"]\ar[d,"m"]&Q\ar[d,"\pi"]\\
    Q\ar[r,"\pi"]&B_X
\end{tikzcd}
\]
By the base change for $R\pi_!$, we have:
\[
\pi^*R\pi_!\Q_\ell[2q](q)\simeq Rm_!\Q_\ell[2q](q).
\]
Hence, the group multiplication $m:Q\times_{B_X}Q\to Q$ induces the multiplication on $\mathcal A$:
\[
 \mathcal A\otimes\mathcal A\simeq R\pi_!(\Q_{\ell}[2q](q)\otimes \pi^*R\pi_!\Q_\ell[2q](q))
 \simeq R\pi_!Rm_!\mathbb{Q}_\ell[4q](2q)
 \xrightarrow{\,R\pi_!(\operatorname{Tr}_m)[2q](q)\,}
 \mathcal A,
\]
where the first map is the K\"unneth isomorphism. The unit of the algebra is induced by the identity section $e:B_X\to Q$.

The homomorphisms $u:Q\to P_X$ and $v:Q\to a^*P_Y$ induce actions of $Q$ on $X$ and $Y_{a}$, respectively.
For $(Z,h_Z)=(X,h_X)$ or $(Y_{a},h_{Y_a})$, write $\alpha_Z:Q\times_{B_X}Z\to Z$ for this action and
$p_Z:Q\times_{B_X}Z\to Z$ for projection.

The isomorphism $(g,z)\mapsto(g,gz)$ identifies $\alpha_Z$ with
$p_Z$, so it is smooth of relative dimension $q$ and we have the following Cartesian diagram:
\[
\begin{tikzcd}
    Q\times_{B_X}Z\ar[r,"p"]\ar[d,"\alpha_Z"]&Q\ar[d,"\pi"]\\
    Z\ar[r,"h_Z"]&B_X
\end{tikzcd}
\]
Similarly, the action by $Q$ gives a natural module structure:
\begin{equation}\label{eq:module structure}
c_Z: \mathcal A\otimes Rh_{Z*}\mathbb{Q}_{\ell,Z}\simeq Rh_{Z!}(h_Z^*\mathcal A\otimes\Q_{\ell,Z})\simeq Rh_{Z!}R\alpha_{Z!}\mathbb{Q}_{\ell,Q\times_{B_X}Z}[2q](q)\xrightarrow{\,Rh_{Z!}(\operatorname{Tr}_{\alpha_Z})\,}Rh_{Z*}\mathbb{Q}_{\ell,Z}.
\end{equation}
These are the cap-product actions on $K_X$ and $a^*K_Y$.
\begin{lemma}\label{lem:weak-actual-action}
These operations make $\mathcal A$ a bounded constructible
algebra and $a^*K_Y,K_X$ into $\mathcal A$-modules.
The natural morphism $\beta:a^*K_Y\to K_X$ is a morphism
of these $\mathcal A$-modules.
\end{lemma}

\begin{proof}
The six-functor formalism of
\cite[Theorems~7.2.6--7.2.7]{LZ} applies to the
separated algebraic space $Q$.
Since $\pi$ is of finite presentation,
\cite[Proposition~7.4.1]{LZ} implies that $\mathcal A$
is constructible. For a geometric point $b$ of $B_X$, the base-change theorem
for $R\pi_!$ identifies $\mathcal A_b$ with
$R\Gamma_c(Q_b,\mathbb{Q}_\ell)[2q](q)$
\cite[Theorem~7.2.7(2)]{LZ}.
Since $Q_b$ has dimension $q$, its compactly supported
cohomology vanishes outside $[0,2q]$.
Thus $\mathcal A$ is concentrated in degrees $[-2q,0]$.

Since we use the cap-product construction of
\cite[Sections~7.4.2--7.4.3]{Ngo} here, its associativity
for equivariant coefficients is proved in
\cite[Proposition~1.4]{MSChi}. It remains to show that the natural morphism $\beta:a^*K_Y\to K_X$ respects these module structures. In other words,
we need to show the commutativity of the following diagram:
\[
\begin{tikzcd}
    \mathcal A\otimes a^*K_Y\ar[r,"\simeq"]\ar[d]&\mathcal A\otimes Rh_{Y_a*}\Q_{\ell,Y_a}\ar[r,"1\otimes\beta"]\ar[d,"c_{Y_a}"]&\mathcal A\otimes K_X\ar[d,"c_X"]\\
    a^*K_Y\ar[r,"\simeq"]&K_{Y_a}:=Rh_{Y_a*}\Q_{\ell,Y_a}\ar[r,"\beta"]&K_X
\end{tikzcd}
\]
By the definition of the module structure, $c_Z$ can be decomposed as follows, see \eqref{eq:module structure}, we have:
\[
\begin{tikzcd}
    \mathcal A\otimes Rh_{Y_a*}\Q_{\ell,Y_a}\ar[r,"1\otimes\beta"]\ar[d,"\simeq"]&\mathcal A\otimes K_X\ar[d,"\simeq"]\\
    Rh_{Y_a*}(h_{Y_a}^*R\pi_!\Q_{\ell,Q}[2q](q))\ar[d,"\simeq"]&Rh_{X*}(h_X^*R\pi_!\Q_{\ell,Q}[2q](q))\ar[d,"\simeq"]\\
     Rh_{Y_a*}(R\alpha_{Y_a!}\Q_{\ell}[2q](q))\ar[d,"Rh_{Y_a*}\operatorname{Tr}"]\ar[r,dotted]&Rh_{X*}(R\alpha_{X!}\Q_{\ell}[2q](q))\ar[d,"Rh_{X*}\operatorname{Tr}"]\\
    K_{Y_a}:=Rh_{Y_a*}\Q_{\ell,Y_a}\ar[r,"\beta"]&K_X
\end{tikzcd}
\]
Hence, it suffices to show the commutativity of the lowest square which follows from the functoriality of the trace map. More precisely, apply the base-change isomorphism to the Cartesian square of action morphisms:
\[
\begin{tikzcd}
    Q\times_{B_X}X\ar[d,"\alpha_X"]\ar[r]&Q\times_{B_X}Y_{a}\ar[d,"\alpha_{Y_{a}}"]\\
    X\ar[r,"F_a"]&Y_{a}
\end{tikzcd}
\]
we obtain
\[
 F_a^*R\alpha_{Y_{a}!}\mathbb{Q}_\ell[2q](q)
 \simeq R\alpha_{X!}\mathbb{Q}_\ell[2q](q),
\]
compatibility of smooth trace with base change gives $F_a^*(\operatorname{Tr}_{\alpha_{Y_{a}}})
 =\operatorname{Tr}_{\alpha_X}$;
see \cite[Remark~4.1.6(2), Theorem~7.2.6]{LZ}. 

The $Q$-equivariance of $F_a$ gives
\[
 F_a((g_1\cdots g_n)x)
 =(g_1\cdots g_n)F_a(x)
 \qquad(n\ge1).
\]
These identities are compatible with multiplying adjacent $g_i$ and inserting the identity of $Q$. Applying the enhanced trace and base-change formalism of \cite[Remark~4.1.6, Theorems~7.2.6--7.2.7]{LZ} therefore makes $\beta$ a morphism of $\mathcal A$-modules in the stable $\infty$-category.

\end{proof}

\begin{lemma}\label{lem:weak-adams-algebra}
The homology algebra of Lemma~\ref{lem:weak-actual-action}
admits a natural equivalence of algebra objects
\[
 \mathcal A\simeq\bigoplus_{k\ge0}\Lambda^kT(Q)[k]
\]
in the symmetric monoidal stable $\infty$-category,
functorial in $Q$.
\end{lemma}
\begin{proof}
Commutativity of $Q$ and the enhanced K\"unneth and trace
formalism equip $\mathcal A$ with a commutative algebra
structure; see \cite[Theorems~7.2.6--7.2.7]{LZ}.
Fix $N\ge2$ and write $\psi=[N]_*$.
By \cite[Sections~7.4.3--7.4.4]{Ngo},
$\mathcal H^{-k}(\mathcal A)=\Lambda^kT(Q)$ as a graded
algebra, with $\psi$ acting by $N^k$.

For $m\ge\max\{1,2q\}$ and $0\le j\le m$, put
\[
 p_j(t)=\prod_{\substack{0\le k\le m\\k\ne j}}(t-N^k),
 \qquad
 \mathcal A_j=\operatorname*{colim}
 \bigl(\mathcal A\xrightarrow{p_j(\psi)}\mathcal A
 \xrightarrow{p_j(\psi)}\cdots\bigr),
\]
where the colimit is taken in the derived $\infty$-category
of sheaves of $\mathbb Q_\ell$-vector spaces on the
pro-\'etale site of $B_X$.
Since filtered colimits are exact and $p_j(N^k)$ vanishes
precisely for $k\ne j$, the canonical maps give
\[
 \mathcal A\xrightarrow{\sim}\bigoplus_{j=0}^m\mathcal A_j,
 \qquad
 \mathcal A_j\simeq\mathcal H^{-j}(\mathcal A)[j].
\]
Thus the $j=1$ summand defines
$T(Q)[1]\to\mathcal A$, inducing the identity on
$\mathcal H^{-1}$.
Group homomorphisms commute with $\psi$, and enlarging
$m$ only introduces factors invertible on the surviving
cohomology. Hence this inclusion is functorial in $Q$.

By the universal property of the free commutative algebra,
the inclusion extends to
\[
 \operatorname{Sym}(T(Q)[1])
 \simeq\bigoplus_{k\ge0}\Lambda^kT(Q)[k]
 \longrightarrow\mathcal A;
\]
see \cite[Proposition~3.1.3.13 and
Example~3.1.3.14]{HA}.
Indeed, permutations of the shifted factors act on
$T(Q)^{\otimes k}[k]$ with the sign character.
Since $k!$ is invertible in $\mathbb Q_\ell$, this gives
$\operatorname{Sym}^k(T(Q)[1])\simeq\Lambda^kT(Q)[k]$.
The resulting algebra morphism induces the exterior-product
isomorphism on every cohomology sheaf, hence is an
equivalence, functorially in $Q$.
\end{proof}

Lemma~\ref{lem:weak-actual-action} gives $ [\beta]\in\Hom_{D(\mathcal A)}(a^*K_Y,K_X)$. Put $F_n^YK_Y=a^*({}^p\tau_{\le n}K_Y)$ and $F_n^XK_X={}^p\tau_{\le n}K_X$.
We regard filtrations as objects of
\[
\operatorname{Seq}(\mathscr D(B_X))
=
\operatorname{Fun}((\mathbb Z,\le),\mathscr D(B_X)),
\]
equipped with the Day-convolution tensor product. The underlying object of a sequence is its colimit.

Put $F_\ell\mathcal A=\tau^{\mathrm{ord}}_{\le\ell}\mathcal A$. This is an algebra object in $\operatorname{Seq}(\mathscr D(B_X))$. By a filtered $\mathcal A$-module we mean a module over $F_\bullet\mathcal A$ in this category.

A filtered $\mathcal A$-linear lift of the fixed enhanced comparison $\beta$ of degree $s$ is a morphism
\[
\widetilde\beta_s:
F_\bullet^YK_Y\longrightarrow F_{\bullet+s}^XK_X
\]
of filtered $\mathcal A$-modules, together with an identification of its colimit with the given $\mathcal A$-module morphism $\beta$.

All the truncation filtrations used here vanish in sufficiently negative degrees, and their colimits recover the corresponding unfiltered objects.

\begin{lemma}\label{lem:weak-initial-lift}
The cap-product actions equip $F_\bullet^YK_Y$ and $F_\bullet^XK_X$ with filtered $\mathcal A$-module structures.
The comparison $\beta$ admits a filtered $\mathcal A$-linear lift of degree $d_X$:
\[
\widetilde\beta_{d_X}:
F_\bullet^YK_Y\longrightarrow F_{\bullet+d_X}^XK_X.
\]
\end{lemma}

\begin{proof}
Write $\mathcal A_{P_Z}$ for the homology algebra of $P_Z/B_Z$, with its ordinary truncation filtration.
Tensoring by an ordinary sheaf preserves ${}^pD^{\le n}$, so
\[
F_\ell\mathcal A_{P_Z}\otimes\ptau_{\le n}K_Z
\in{}^pD^{\le n+\ell}(B_Z),
\qquad Z=X,Y.
\]
The cap-product map therefore factors through $\ptau_{\le n+\ell}K_Z$. Restrict the $X$-action along $Q\to P_X$; for $Y$, first pull back from $B_Y$ and then restrict along $Q\to a^*P_Y$.

Also,
\[
a^*(\ptau_{\le n}K_Y)
\in D^{\mathrm{ord},\le n}(B_X)
\subset{}^pD^{\le n+d_X}(B_X),
\]
so the composite with $\beta$ factors through $\ptau_{\le n+d_X}K_X$.

These factorizations retain the enhanced module structures. Indeed, truncation adjunction gives
\[
\operatorname{Map}_{\mathscr D(B)}
(C,\ptau_{\le m}K)
\simeq
\operatorname{Map}_{\mathscr D(B)}(C,K)
\qquad(C\in{}^pD^{\le m}).
\]
The degree bounds above apply to all iterated actions, so the actions and the comparison factor through the
truncations together with their coherences. For sufficiently large $n$, the truncations are the full
complexes and the resulting comparison is $\beta$.
\end{proof}

\begin{remark}
The lemma gives an initial truncation bound. To obtain compatibility in \eqref{eq:compatibility}, we
will improve the degree from $d_X$ to $d_X-d_Y$ while retaining the filtered $\mathcal A$-module structure.
This improvement is obtained by the vanishing of the corresponding obstruction groups.
\end{remark}

\subsection{Geometric bounds and the free exterior factor}
Set
\[
 \mathcal P_i=\pH^{d_Y+i}K_Y,\qquad
 \mathcal R_j=\pH^{d_X+j}K_X,\qquad \Delta=d_Y-d_X.
\]
Choose a finite stratification of $B_X$ by smooth locally
closed subvarieties $S$ such that $r(b)$ is constant on
each $S$, and the restrictions of the homology sheaves
of $Q$ and the sheaves
\[
 U_i^u=\mathcal H^u(\iota^*a^*\mathcal P_i),
 \qquad
 V_j^v=\mathcal H^v(\iota^!\mathcal R_j),
 \qquad \iota:S\hookrightarrow B_X,
\]
are local systems on $S$ for all $i,j,u,v$.
Write $t=\dim S$.
Here $i,j$ are perverse indices, and $u,v$
are ordinary cohomological degrees.
\begin{lemma}\label{lem:weak-stalk-costalk}
On every smooth stratum $S$ of dimension $t$, the ordinary
coefficient systems satisfy
\begin{equation}\label{eq:rhm-stalk-costalk}
 U_i^u\ne0\Rightarrow d_Y+u\le i\le d_Y-u,\qquad
 V_j^v\ne0\Rightarrow d_X-v-2t\le j\le d_X+v+2t.
\end{equation}
\end{lemma}
\begin{proof}

Fix $b\in S$. Since $\mathcal P_i[-d_Y-i]$ is a direct
summand of $K_Y$, proper base change identifies $(U_i^u)_b$
with a direct summand of
$H^{d_Y+i+u}(Y_{a(b)},\mathbb Q_\ell)$.
The fiber dimension bound therefore gives $i+u\le d_Y$
whenever $(U_i^u)_b\ne0$. Relative hard Lefschetz gives
$\mathcal P_i\simeq\mathcal P_{2d_Y-i}(d_Y-i)$.
Applying the same bound  yields
$2d_Y-i+u\le d_Y$, hence $i\ge d_Y+u$.

Since $\mathbb D K_X\simeq K_X[4d_X](2d_X)$, we have
$\mathbb D\mathcal R_j\simeq
\mathcal R_{2d_X-j}(2d_X)$.
Verdier duality on the smooth stratum $S$, where the
cohomology sheaves are locally constant, gives
\[
 (V_j^v)_b\simeq
 \left(
 \mathcal H^{-v-2t}
 (\iota^*\mathcal R_{2d_X-j})_b
 \right)^\vee(t-2d_X).
\]
Thus $(V_j^v)_b\ne0$ implies
$2d_X-j\ge d_X-v-2t$ by the stalk bound for $X$,
equivalently $j\le d_X+v+2t$. Apply the relative hard Lefschetz again, $\mathcal R_j\simeq \mathcal R_{2d_X-j}(d_X-j)$, then we have $j\ge d_X-v-2t$.
\end{proof}

For $b\in S$, let
\[
 f_b:T(Q)_b\xrightarrow{T(u)_b}T(P_X)_b
 \longrightarrow T(\operatorname{Ab}P_{X,b}),
 \qquad
 r=\frac12\dim_{\mathbb{Q}_\ell}\operatorname{im}f_b.
\]
Fix a subspace $A_Q\subset T(Q)_b$ such that
$f_b|_{A_Q}:A_Q\xrightarrow{\sim}\operatorname{im}f_b$.

\begin{proposition}\label{prop:weak-costalk-free}
Let $A_Q\subset T(Q)_b$ as above with $\dim A_Q=2r$.
For every ordinary cohomological degree $v$, the graded module $\bigoplus_j (V_j^v)_b$ is free over $\Lambda A_Q$ under the action induced by Lemma~\ref{lem:weak-actual-action}.
\end{proposition}
\begin{proof}

Since $S$ is smooth of dimension $t$, by Verdier duality,
\[
 \mathcal H^v(\iota^!\mathcal R_j)_b
 \simeq
 \left(
 \mathcal H^{-v-2t}
 (\iota^*\mathcal R_{2d_X-j})_b
 \right)^\vee(t-2d_X).
\]

Fix a lift $\lambda:T(\operatorname{Ab}P_{X,b})\to T(P_{X,b})$
and put $A=\operatorname{im}\lambda$.
Each vector $\xi\in A$ is a generator of $\Lambda A$
in grading degree $-1$.
The costalk operator
\[
 \xi:\mathcal H^v(\iota^!\mathcal R_j)_b
 \longrightarrow
 \mathcal H^v(\iota^!\mathcal R_{j-1})_b
\]
corresponds under Verdier duality to the transpose of
\[
 \xi:\mathcal H^{-v-2t}(\iota^*\mathcal R_{2d_X-j+1})_b
 \longrightarrow
 \mathcal H^{-v-2t}(\iota^*\mathcal R_{2d_X-j})_b,
\]
with the usual grading signs and the sign
$\xi\mapsto-\xi$ induced by group inversion.
Thus the costalk module carries the graded dual action
of the stalk module.

By \cite[Propositions~7.4.10, 7.5.5, 7.7.4 and
Lemma~7.7.5]{Ngo} (see also \cite[Proposition~1.5 and Section~1.6, Step~B]{MSChi}), $\bigoplus_j\mathcal H^m(\iota^*\mathcal R_j)_b$ is a free $\Lambda A$-module. Writing $m=\dim A$, the perfect pairing $\Lambda^kA\otimes\Lambda^{m-k}A\to\Lambda^mA$ shows that the graded dual of a finite free $\Lambda A$-module is free. Twisting the action by $\xi\mapsto-\xi$ also preserves
freeness. Hence $\bigoplus_j\mathcal H^v(\iota^!\mathcal R_j)_b$ is free over $\Lambda A$.

Recall that $u:Q\to P_X$ is the homomorphism in our assumptions.
Since $T(u)_b(A_Q)$ projects injectively to
$T(\operatorname{Ab}P_{X,b})$, we may take the subspace
$A\subset T(P_X)_b$ in the preceding argument to satisfy
\[
 T(u)_b(A_Q)\subset A,
 \qquad
A\xrightarrow{\sim}T(\operatorname{Ab}P_{X,b}).
\]
Naturality of the cap product action identifies this restriction with the action induced by $Q$. Since $\Lambda A$ is free over $\Lambda A_Q$, the freeness then follows.
\end{proof}

\begin{remark}\label{rmk:weak-costalk-free-IC}
Proposition~\ref{prop:weak-costalk-free} also holds without
the RHM assumption if $K_X$ is replaced by
$Rh_{X*}\mathrm{IC}_X[-\dim X]$.
Indeed, the strict-support freeness also holds for
equivariant IC coefficients
\cite[Propositions~1.4--1.6]{MSChi}. The arguments for the costalk and the restriction along $Q\to P_X$ are unchanged.
\end{remark}

\subsection{From local degree bounds to global Ext vanishing}
\label{subsec:global-ext}
Put
\[
\operatorname{Gr}(\mathscr D(B_X))
=
\operatorname{Fun}(\mathbb Z_{\mathrm{disc}},\mathscr D(B_X)),
\]
with its Day-convolution tensor product. For an algebra object $\mathcal B$ in this category, define
\[
D_{\mathrm{gr}}(\mathcal B)
=
h\operatorname{Mod}_{\mathcal B}
\bigl(\operatorname{Gr}(\mathscr D(B_X))\bigr).
\]

The associated graded construction is given by
\[
(\operatorname{gr}F)_n
=
\operatorname{cofib}(F_{n-1}\to F_n).
\]
It is symmetric monoidal for the relevant convolution products; see \cite[Section~2.23 and Proposition~2.26]{GP}.
It therefore sends filtered algebras and modules to graded algebras and modules.

When $\mathcal B$ is concentrated in cohomological degree zero, $D_{\mathrm{gr}}(\mathcal B)$ identifies with the
derived category of sheaves of internally graded $\mathcal B$-modules. For a graded complex $C$, use the shearing convention
\[
(\mathsf SC)_n=C_n[n],
\]
with the usual Koszul signs.

Since
\[
\operatorname{gr}_{-k}^F\mathcal A
\simeq\Lambda^kT(Q)[k],
\]
the sheared associated graded algebra has cohomology concentrated in degree zero and is quasi-isomorphic to
\[
E=\Lambda T(Q),
\qquad E_{-k}=\Lambda^kT(Q).
\]

Shearing and reindexing the associated graded modules give objects $U,V\in D_{\mathrm{gr}}(E)$ with components
\[
\begin{aligned}
U_i
&=(\operatorname{gr}_{d_Y+i}^F F^YK_Y)[d_Y+i]
 \simeq a^*\mathcal P_i,\\
V_j
&=(\operatorname{gr}_{d_X+j}^F F^XK_X)[d_X+j]
 \simeq\mathcal R_j.
\end{aligned}
\]
with the $E$-module structures induced by the filtered cap-product actions. In particular, their stalk and costalk cohomology modules are precisely the $U_i^u$ and $V_j^v$ considered above.

Write $\cRHom_B^{\mathrm{gr}}(L,M)$ for the internal derived Hom of graded $B$-modules, preserving internal degree, and put
$\cExt_B^{e,\mathrm{gr}}(L,M)
=\mathcal H^e\cRHom_B^{\mathrm{gr}}(L,M)$. Over a base $W$, global derived Hom is
\[
\RHom_B^{\mathrm{gr}}(L,M)
=R\Gamma\bigl(W,\cRHom_B^{\mathrm{gr}}(L,M)\bigr).
\]

In particular,
\[
\Hom_{D_{\mathrm{gr}}(\mathcal B)}(L,M[e])
=
H^eR\Gamma\bigl(
B_X,\cRHom_{\mathcal B}^{\mathrm{gr}}(L,M)
\bigr).
\]
The superscript $\mathrm{gr}$ requires preservation of internal degree. The shifts $[e]$ and $\langle\delta\rangle$ refer to cohomological and internal degree, respectively. Finally, define
\[
(V\langle\delta\rangle)_i=V_{i+\delta},
\qquad
\Ext^e_{E,\delta}(U,V)
=
\Hom_{D_{\mathrm{gr}}(E)}
\bigl(U,V\langle\delta\rangle[e]\bigr).
\]

\begin{lemma}\label{lem:weak-degree-bound}
On a pro-\'etale cover of a stratum $S$ of dimension $t$ trivializing $T(Q)|_S$ and the graded local systems $U^u=\bigoplus_i U_i^u$ and $V^v=\bigoplus_j V_j^v$, a nonzero morphism
$f:U^u\to V^v\langle\lambda\rangle$
of graded $\Lambda A_Q$-modules satisfies
\begin{equation}\label{eq:rhm-degree-bound}
 \lambda\le-\Delta+v-u+2(t-r).
\end{equation}
Here $\Delta=d_Y-d_X$, $\dim A_Q=2r$.
\end{lemma}
\begin{proof}
On such a pro-\'etale cover of $S$, split
$T(Q)=A_Q\oplus R_0$ as vector spaces. Let $f:U^u\to V^v\langle\lambda\rangle$ be a nonzero
morphism of graded $\Lambda A_Q$-modules, with components
$f_i:U_i^u\to V_{i+\lambda}^v$.
Since the augmentation ideal of $\Lambda A_Q$ is nilpotent,
there exist an index $i$ and an element $x\in U_i^u$
such that
\[
 0\ne y:=f_i(x)\in V_{i+\lambda}^v,
 \qquad
 \xi y=0\quad\text{for every }\xi\in A_Q.
\]  We now bound the degrees of $x$ and $y$.
Choose a basis $a_1,\ldots,a_{2r}$ of $A_Q$.
Each $a_k$ acts by maps $V_j^v\to V_{j-1}^v$.
By Proposition~\ref{prop:weak-costalk-free}, choose
free generators $e_\alpha\in V_{h_\alpha}^v$ such that
\[
 V^v=\bigoplus_\alpha(\Lambda A_Q)e_\alpha.
\]
Since $e_\alpha\ne0$,
Lemma~\ref{lem:weak-stalk-costalk} gives $h_\alpha\le d_X+v+2t$.

In $\Lambda A_Q$, an element is annihilated by every
$a_k$ if and only if it is a scalar multiple of
$a_1\cdots a_{2r}$.
Consequently, the relations $a_k y=0$ and freeness imply $y=\sum_\alpha c_\alpha\cdot(a_1\cdots a_{2r})e_\alpha$.
Each nonzero term belongs to $V_{h_\alpha-2r}^v$.
Since $0\ne y\in V_{i+\lambda}^v$, at least one
$c_\alpha$ is nonzero, and for every such $\alpha$,
\[
 i+\lambda=h_\alpha-2r
 \le d_X+v+2t-2r.
\] 

On the other hand, $f_i(x)=y\ne0$ implies $U_i^u\ne0$,
so Lemma~\ref{lem:weak-stalk-costalk} gives
$i\ge d_Y+u$.
Combining these inequalities yields
\[
 \lambda\le(d_X+v+2t-2r)-(d_Y+u)
 =-\Delta+v-u+2(t-r),
\]
which proves \eqref{eq:rhm-degree-bound}.

\end{proof}

\begin{proposition}\label{prop:weak-global-ext}
For the objects $U,V\in D_{\mathrm{gr}}(E)$ defined above, the effective-rank
inequality~\ref{assum:effective-rank inequality}, i.e., $r\geq t$, on every stratum implies
\[
 \Ext^e_{E,\delta}(U,V)=0\qquad \text{if} \ \ e<\Delta+\delta .
\]
More generally, if $r\ge t-\kappa$ on every stratum, the vanishing
range is $e<\Delta+\delta-2\kappa$. 
\end{proposition}
\begin{proof}
Fix a stratum $S$, passing to the pro-\'etale covering of $S$ if necessary, we may assume $T(Q)=A_Q\oplus R_0$. Then $E=\Lambda A_Q\otimes_{\Q_\ell}\Lambda R_0$, and the tensor product is graded: for $\xi\in R_0$
and $a\in \Lambda A_Q$ of degree $|a|$, $\xi a=(-1)^{|a|}a\xi$.  All Hom and Ext operations in the following local calculation are internal and preserve the grading.

For each fixed $v$, Proposition~\ref{prop:weak-costalk-free} makes $V^v$ a free graded $\Lambda A_Q$-module.
For fixed $u,v$, Lemma~\ref{lem:weak-degree-bound} gives
\[
 \cHom_{\Lambda A_Q}^{\mathrm{gr}}
 \bigl(U^u,V^v\langle\lambda\rangle\bigr)=0
 \qquad
 \text{if }\lambda>-\Delta+v-u+2(t-r).
\] 
After refining the pro-\'etale cover to trivialize $T(Q)$, $U^u$ and $V^v$, internal Ext is computed by sheafifying resolutions of $\Lambda A_Q$-modules by finite free graded modules. Since $\Lambda A_Q$ is a finite-dimensional Frobenius
algebra, every finite free graded $\Lambda A_Q$-module
is injective in the category of graded modules.
Thus
\[
 \cExt_{\Lambda A_Q}^{q,\mathrm{gr}}
 \bigl(U^u,V^v\langle\lambda\rangle\bigr)=0
 \qquad(q>0).
\]

Define a graded sheaf $H=\bigoplus_\lambda H_\lambda$ by
\[
 H_\lambda=
 \cHom_{\Lambda A_Q}^{\mathrm{gr}}
 \bigl(U^u,V^v\langle\lambda\rangle\bigr).
\]
Thus a local section $f\in H_\lambda$ is a family of maps $f_i:U_i^u\longrightarrow V_{i+\lambda}^v$ satisfying $f(am)=af(m)$.

For $\xi\in R_0$, define its action on $H$ by
\[
 (\xi\star f)(m)
 =
 (-1)^i\bigl(\xi_Vf(m)-f(\xi_Um)\bigr),
 \qquad m\in U_i^u,
\]
which defines a $\Lambda A_Q$-linear map. Here the subscripts in $\xi_U,\xi_V$ indicate the module on which $\xi$ acts. The operators satisfy $\xi\star(\eta\star f)+\eta\star(\xi\star f)=0$, hence make $H$ a graded $\Lambda R_0$-module. 

The condition $\xi\star f=0$ for every $\xi\in R_0$
is equivalent to $\xi_Vf(m)=f(\xi_Um)$.
Together with $\Lambda A_Q$-linearity, this is precisely
$E$-linearity. Consequently,
\[
 \cHom_{E}^{\mathrm{gr}}
 \bigl(U^u,V^v\langle\delta\rangle\bigr)
 \simeq
 \cHom_{\Lambda R_0}^{\mathrm{gr}}
 \bigl(\mathbb{Q}_\ell,H\langle\delta\rangle\bigr),
\]
where $\mathbb{Q}_\ell$ is concentrated in degree zero
with trivial $R_0$-action.

We have an identity of composite functors:
\[
 \cHom_{E}^{\mathrm{gr}}(U^u,-)
 \simeq
 \cHom_{\Lambda R_0}^{\mathrm{gr}}
 (\mathbb{Q}_\ell,-)
 \circ
 \left(
 \bigoplus_{\lambda\in\mathbb Z}
 \cHom_{\Lambda A_Q}^{\mathrm{gr}}
 \bigl(U^u,(-)\langle\lambda\rangle\bigr)
 \right).
\]

The inner functor carries the $\Lambda R_0$-action defined above.
It is right adjoint to the exact functor $W\longmapsto U^u\otimes_{\mathbb{Q}_\ell}W$,
with $E$-action
\[
 a(m\otimes w)=am\otimes w,
 \qquad
 \xi(m\otimes w)
 =\xi_Um\otimes w+(-1)^i m\otimes\xi_Ww,
 \qquad m\in U_i^u.
\]
It therefore sends injective graded $E$-modules
to injective graded $\Lambda R_0$-modules.
Since $E$ is free over $\Lambda A_Q$, restriction to $\Lambda A_Q$ also preserves injective modules.
Hence we have the Grothendieck spectral sequence
\[
 \begin{aligned}
 E_2^{p,q}
 &=
 \cExt_{\Lambda R_0}^{p,\mathrm{gr}}
 \left(
 \mathbb{Q}_\ell,
 \left[
 \bigoplus_{\lambda\in\mathbb Z}
 \cExt_{\Lambda A_Q}^{q,\mathrm{gr}}
 \bigl(U^u,V^v\langle\lambda\rangle\bigr)
 \right]\langle\delta\rangle
 \right)
 \\
 &\Longrightarrow
 \cExt_{E}^{p+q,\mathrm{gr}}
 \bigl(U^u,V^v\langle\delta\rangle\bigr).
 \end{aligned}
\]
By the internal Ext vanishing above, all rows with $q>0$ vanish. Thus
\[
 \cExt_{E}^{n,\mathrm{gr}}
 \bigl(U^u,V^v\langle\delta\rangle\bigr)
 \simeq
 \cExt_{\Lambda R_0}^{n,\mathrm{gr}}
 \bigl(\mathbb{Q}_\ell,H\langle\delta\rangle\bigr)
 \qquad(n\ge0).
\]
Thus these Ext sheaves are computed by applying $\cHom_{\Lambda R_0}^{\mathrm{gr}}(-,H\langle\delta\rangle)$ to the Koszul resolution of $\mathbb{Q}_\ell$.

The $n$th term of the Koszul resolution is $P_n=\Lambda R_0\otimes_{\mathbb Q_\ell}\operatorname{Sym}^nR_0$ with free generators in degree $-n$.
A degree-preserving map $P_n\to H\langle\delta\rangle$
sends these generators to $H_{\delta-n}$, whose elements
are $\Lambda A_Q$-linear maps $U_i^u\longrightarrow V_{i+\delta-n}^v$. By Lemma~\ref{lem:weak-degree-bound}, 
\[\cHom_{\Lambda R_0}^{\mathrm{gr}}
 \bigl(P_n,H\langle\delta\rangle\bigr)=0, \qquad \text{if}\ \ \delta-n>-\Delta+v-u+2(t-r).
\]
Taking cohomology and using the preceding
Ext isomorphism gives
\[
\cExt_{E}^{n,\mathrm{gr}} \bigl(U^u,V^v\langle\delta\rangle\bigr)=0
 \qquad
 \text{if}\ \ n<\Delta+\delta+u-v+2(r-t)
\]

We reduce the global calculation to the strata.
Let
\[
 \iota:Z\hookrightarrow B_X,\qquad
 j:O=B_X\setminus Z\hookrightarrow B_X
\]
be a closed inclusion and its open complement.
The $E$-action on $V$ induces an $\iota^*E$-action
on $\iota^!V$. In fact,
\[
 \iota_*(\iota^*E\otimes\iota^!V)
 \simeq E\otimes\iota_*\iota^!V
 \longrightarrow E\otimes V
 \longrightarrow V.
\]
Pushing forward this action gives an $E$-module structure
on $\iota_*\iota^!V$ through
\[
 E\otimes\iota_*\iota^!V
 \simeq
 \iota_*(\iota^*E\otimes\iota^!V)
 \longrightarrow
 \iota_*\iota^!V.
\]
Similarly, $Rj_*j^*V$ is also an $E$-module.

Thus the localization triangle
 $\iota_*\iota^!V\longrightarrow V
 \longrightarrow Rj_*j^*V\xrightarrow{+1}$
is a triangle of $E$-modules. 
Applying $\RHom_E^{\mathrm{gr}}(U,-)$ and using adjunction gives a triangle
\[
 \RHom_{\iota^*E}^{\mathrm{gr}}
   (\iota^*U,\iota^!V)
 \longrightarrow \RHom_E^{\mathrm{gr}}(U,V)
 \longrightarrow
 \RHom_{j^*E}^{\mathrm{gr}}(j^*U,j^*V)
 \xrightarrow{+1}.
\]
Induction over strata reduces the desired
vanishing to the terms
\[
 \RHom_{E_S}^{\mathrm{gr}}(i_S^*U,i_S^!V),
 \qquad
 i_S:S\hookrightarrow B_X,\quad E_S=i_S^*E.
\]

We already get the vanishing $\cExt_{E_S}^{n,\mathrm{gr}}
(U^u,V^v\langle\delta\rangle)=0$  if  $n<\Delta+\delta+u-v+2(r-t)$. 
Since $r\ge t$, the vanishing implies that for all pairs $(u,v)$,
\[
 \mathcal H^e \cRHom_{E_S}^{\mathrm{gr}}
 \bigl(U^u[-u],V^v[-v]\langle\delta\rangle\bigr)=
 \cExt_{E_S}^{e+u-v,\mathrm{gr}}
 \bigl(U^u,V^v\langle\delta\rangle\bigr)=0
 \qquad\text{if}\ \ e<\Delta+\delta.
\]

The $E$-module truncation triangles
\[
 \tau_{\le v-1}(i_S^!V)
 \longrightarrow \tau_{\le v}(i_S^!V)
 \longrightarrow V^v[-v]\xrightarrow{+1}
\]
and the corresponding triangles for $i_S^*U$
are finite.
Applying internal derived Hom and inducting in both variables
therefore gives
\[
 \cRHom_{E_S}^{\mathrm{gr}}
 \bigl(i_S^*U,i_S^!V\langle\delta\rangle\bigr)
 \in D^{\ge\Delta+\delta}.
\]
Taking derived global sections preserves this lower bound.
The module localization triangles then give
\[
 \Ext^e_{E,\delta}(U,V)=0
 \qquad\text{if}\ \ e<\Delta+\delta
\]
on the whole base.

If $r\ge t-\kappa$, the same argument gives
\[
 \Ext^e_{E,\delta}(U,V)=0
 \qquad\text{if}\ \ e<\Delta+\delta-2\kappa.
\]
\end{proof}

\subsection{Obstruction to finer filtered morphisms}
Let
\[
\widetilde\beta_s:
F_\bullet^YK_Y\longrightarrow F_{\bullet+s}^XK_X
\]
be a filtered $\mathcal A$-linear lift of the fixed comparison $\beta$.
We improve its degree while retaining the filtered module structure and the identification of its underlying
morphism with $\beta$. Put $\Delta=d_Y-d_X$.
\begin{lemma}\label{lem:weak-rees-lowering}
The obstruction to factoring $\widetilde\beta_s$ through the transition morphism
\[
F_{\bullet+s-1}^XK_X\longrightarrow F_{\bullet+s}^XK_X
\]
in the category of filtered $\mathcal A$-modules is a class in
\[
\operatorname{Ext}^{-s}_{E,\Delta+s}(U,V).
\]
If this group vanishes, there exists a filtered $\mathcal A$-linear lift of $\beta$ of degree $s-1$
whose composite with the transition morphism is $\widetilde\beta_s$.
\end{lemma}
\begin{proof}
Form the Rees algebra and modules
\[
R=\bigoplus_n F_n\mathcal A\,t^n,
M=\bigoplus_n a^*({}^p\tau_{\le n}K_Y)t^n,
N=\bigoplus_n({}^p\tau_{\le n}K_X)t^n.
\]
Multiplication by $t$ is the filtration transition map.
The Rees construction represents the filtered lift by a morphism
\[
\Phi_s\in
\Hom_{D_{\mathrm{gr}}(R)}(M,N\langle s\rangle)
=
H^0R\Gamma\bigl(
B_X,\cRHom_R^{\mathrm{gr}}(M,N\langle s\rangle)
\bigr).
\]
Here $R$ is an internally graded algebra object; its cohomological and internal degrees are distinct. All quotients below are derived; in particular, $\bar R=R\otimes^{\mathbf L}_{\mathbb Q_\ell[t]}\mathbb Q_\ell$,
and likewise $\bar M=M/t$ and $\bar N=N/t$, with $t$ acting by zero
on $\mathbb Q_\ell$. Then
\[
 \bar R_{-k}\simeq\Lambda^kT(Q)[k],\qquad
 \bar M_n\simeq a^*({}^p\mathcal H^nK_Y)[-n],\qquad
 \bar N_n\simeq{}^p\mathcal H^nK_X[-n].
\]
Their sheared algebra and modules are precisely $E,U,V$ defined
above. After forgetting the grading, specialization at $t=1$ gives
\[
 M/(t-1)\simeq a^*K_Y,\qquad
 N/(t-1)\simeq K_X,\qquad \Phi_s|_{t=1}=\beta.
\]

The triangle in $D_{\mathrm{gr}}(R)$
\[
 N\langle s-1\rangle\xrightarrow{t}N\langle s\rangle
 \xrightarrow{q_s}\bar N\langle s\rangle\xrightarrow{+1}
\]
shows that the obstruction to lowering $s$ is
$o(\Phi_s)=[q_s\Phi_s]$.
Derived extension of scalars and shearing give
\[
\cRHom_R^{\mathrm{gr}}(M,\bar N\langle s\rangle)
\simeq\cRHom_{\bar R}^{\mathrm{gr}}(\bar M,\bar N\langle s\rangle)
\simeq\cRHom_E^{\mathrm{gr}}(U,V\langle\Delta+s\rangle[-s]).
\]
Indeed, writing
$n=d_Y+i$, a component of the middle expression is
\[
 a^*({}^p\mathcal H^nK_Y)[-n]
 \longrightarrow{}^p\mathcal H^{n+s}K_X[-n-s],
 \qquad n+s=d_X+(i+\Delta+s),
\]
which becomes $U_i\to V_{i+\Delta+s}[-s]$ after shifting by $[n]$.
Taking $H^0R\Gamma(B_X,-)$ therefore places $o(\beta_s)$ in
$\Ext^{-s}_{E,\Delta+s}(U,V)$.
If this group vanishes, the long exact Hom sequence gives
\[
 \Phi_{s-1}:M\longrightarrow N\langle s-1\rangle,
 \qquad t\Phi_{s-1}=\Phi_s.
\]
Under the Rees correspondence, this gives the required filtered module morphism of degree $s-1$.
Specialization at $t=1$ shows that its underlying comparison is still $\beta$.
\end{proof}

\subsection{Conclusion of the main argument}

\begin{lemma}\label{lem:weak-rational-descent}
Under the geometric hypotheses, suppose that, for some $s\in\mathbb Z$ and every $n\in\mathbb Z$, the
natural comparison with $\mathbb Q_\ell$-coefficients induces factorizations
\[
a^*(\ptau_{\le n}K_{Y,\mathbb Q_\ell})
\longrightarrow
\ptau_{\le n+s}K_{X,\mathbb Q_\ell}
\longrightarrow K_{X,\mathbb Q_\ell}.
\]
Then the corresponding factorizations exist over $\mathbb Q$. They are unique in the ordinary derived
category and commute with the truncation transition maps.
\end{lemma}

\begin{proof}
This follows from the standard comparison argument
of \cite[§6.1]{BBD} and faithful flatness of $\mathbb{Q}_{\ell}/\mathbb{Q}$.
\end{proof}

\begin{proof}[Proof of Theorem~\ref{thm:common-action-weak}]
On a fixed stratum $S$ of dimension $t$ on which $r(b)$ is a constant $r$, the stratum is contained in $\{b:r(b)\le r\}$.
The rank hypothesis therefore gives $t\le r$. Lemmas~\ref{lem:weak-actual-action}--\ref{lem:weak-initial-lift} give the natural $\mathcal A$-module morphism $\beta$ and a filtered $\mathcal A$-linear lift
\[
\widetilde\beta_{d_X}:
F_\bullet^YK_Y\longrightarrow F_{\bullet+d_X}^XK_X.
\]
Lemma~\ref{lem:weak-stalk-costalk} and
Proposition~\ref{prop:weak-costalk-free} provide the geometric inputs
to Lemma~\ref{lem:weak-degree-bound}; Proposition~\ref{prop:weak-global-ext}
then gives $\Ext^e_{E,\delta}(U,V)=0$ for $e<\Delta+\delta$.

For $s>-\Delta$, put $\delta=\Delta+s>0$. The obstruction in
Lemma~\ref{lem:weak-rees-lowering} vanishes because $-s=\Delta-\delta<\Delta+\delta$.

Starting at $s=d_X$, we lower the degree finitely many times to $s=-\Delta=d_X-d_Y$.
Forgetting the module structure and passing to the homotopy category gives factorizations
\[
a^*(\ptau_{\le n}K_{Y,\mathbb Q_\ell})
\longrightarrow
\ptau_{\le n+d_X-d_Y}K_{X,\mathbb Q_\ell}
\longrightarrow K_{X,\mathbb Q_\ell}
\]
of the natural comparison for every $n$.

Lemma~\ref{lem:weak-rational-descent} gives the corresponding factorizations over $\mathbb Q$. Setting $n=d_Y+k$ proves \eqref{eq:rhm-weak-towers}, and hence compatibility in the sense of Definition~\ref{def:weak}. The assertion about the perverse filtration follows by hypercohomology.
\end{proof}

Keep the geometric hypotheses preceding
Theorem~\ref{thm:common-action-weak}, but replace its rank inequality by
\begin{equation}\label{eq:common-rank-defect}
 \dim\{b\in B_X:r(b)\le c\}\le c+\kappa\qquad(c\ge0),
 \qquad\kappa\in\mathbb Z_{\ge0}.
\end{equation}
\begin{corollary}\label{cor:common-action-rank-defect}
The natural comparison morphism admits compatible factorizations
\begin{equation}\label{eq:common-rank-defect-towers}
 a^*\ptau_{\le d_Y+k}K_Y\longrightarrow
 \ptau_{\le d_X+k+\kappa}K_X\qquad(k\in\mathbb Z).
\end{equation}
In particular,
$F^*P_kH^m(Y,\Q)\subseteq P_{k+\kappa}H^m(X,\Q)$.

\end{corollary}

\begin{proof}
On a fixed stratum $S$, $r\ge t-\kappa$, so the proof of
Proposition~\ref{prop:weak-global-ext} gives
\[
\operatorname{Ext}^e_{E,\delta}(U,V)=0
 \qquad(e<\Delta+\delta-2\kappa).
\]
If $\kappa\ge d_Y$, Lemma~\ref{lem:weak-initial-lift}, followed by the canonical truncation transition maps,
gives the factorizations over $\mathbb Q_\ell$. They descend to $\mathbb Q$ by Lemma~\ref{lem:weak-rational-descent}.

Assume now that $0\le\kappa<d_Y$. For $s>\kappa-\Delta$, we have
\[
-s<\Delta+(\Delta+s)-2\kappa.
\]
Lemma~\ref{lem:weak-rees-lowering} therefore successively improves the filtered degree to $\kappa-\Delta$,
giving a filtered $\mathcal A$-linear lift
\[
\widetilde\beta_{\kappa-\Delta}:
F_\bullet^YK_Y
\longrightarrow F_{\bullet+\kappa-\Delta}^XK_X.
\]
Forgetting the module structure and passing to the homotopy category gives the corresponding truncation
factorizations over $\mathbb Q_\ell$. Lemma~\ref{lem:weak-rational-descent} descends these
factorizations to $\mathbb Q$. Writing $n=d_Y+k$ proves \eqref{eq:common-rank-defect-towers}.
\end{proof}

\section{Multiplicative perverse truncations for Hitchin Systems}\label{sec:multiplicativity}

\subsection{The strongly parabolic Hitchin systems}\label{subsec:hitchin-geometry}
Let $G$ be a connected reductive group, $C$ a smooth connected projective curve, and $D\subset C$ a reduced divisor, possibly empty. Put $L=K_C(D)$.  Fix a Borel subgroup $\mathsf B\subset G$ and put
$\mathfrak b=\Lie\mathsf B$, $\mathfrak n=\Lie R_u(\mathsf B)$,
where $R_u(\mathsf B)$ is the unipotent radical. Let $\mathcal G$ be the Iwahori group scheme on $C$ determined by
\[
 \mathcal G|_{C\setminus D}=G\times(C\setminus D),\qquad
 \mathcal G(\widehat{\mathcal O}_{C,p})
 =\{g\in G(\widehat{\mathcal O}_{C,p}):g(p)\in\mathsf B\}
 \quad(p\in D).
\]
For $D=\varnothing$, take $\mathcal G=G\times C$.
Let $\operatorname{Bun}_{\mathcal G}$ be the moduli stack of $\mathcal G$-bundles. For a scheme $S$, an object of $\operatorname{Bun}_{\mathcal G}(S)$ can be interpreted as a principal $G$-bundle $E$ on $C\times S$ together with sections
$\sigma_p:S\longrightarrow(E/\mathsf B)|_{\{p\}\times S}, \ \  p\in D$, see \cite[Section~1, Example~(2)]{Hei10} and \cite[Theorem~4.8 and Example~4.9]{MRR23}. Morphisms are bundle isomorphisms preserving these sections. Set $\ad E:=E\times^G\mathfrak g$, where $\mathfrak g=\Lie G$
and $G$ acts on $\mathfrak g$ by the adjoint representation. Put $E_p:=E|_{\{p\}\times S}$. Pulling back the
$\mathsf B$-torsor $E_p\to E_p/\mathsf B$ along the section $\sigma_p:S\to E_p/\mathsf B$ gives
\[
E_{\mathsf B,p}
:=S\times_{\sigma_p,\,E_p/\mathsf B}E_p.
\] For a geometric object $(E,\sigma)$, define
\[
\mathfrak b_p:=E_{\mathsf B,p}\times^{\mathsf B}\mathfrak b
\subset(\ad E)_p,
\qquad
\mathfrak n_p:=E_{\mathsf B,p}\times^{\mathsf B}\mathfrak n
\subset\mathfrak b_p,
\]
where both contracted products use the adjoint action of $\mathsf B$. With $i_p:\{p\}\hookrightarrow C$, set
\[
\begin{aligned}
 \operatorname{ParAd}(E)
 &=\ker\!\left(\ad E\longrightarrow
       \bigoplus_{p\in D}i_{p*}((\ad E)_p/\mathfrak b_p)\right),\\
 \operatorname{SParAd}(E)
 &=\ker\!\left(\ad E\longrightarrow
       \bigoplus_{p\in D}i_{p*}((\ad E)_p/\mathfrak n_p)\right).
\end{aligned}
\]
For the $\mathcal G$-torsor $\mathcal E$ corresponding to
$(E,\sigma)$, we can also define the adjoint
bundle using the Lie algebra of the group scheme $\mathcal G$. In fact,
$\operatorname{ad}(\mathcal E)=\operatorname{ParAd}(E)$.
\begin{definition}
    A \emph{strongly parabolic Higgs bundle} is a triple $(E,\sigma,\phi)$ with
\begin{equation}\label{eq:strongly-parabolic-field}
 \phi\in H^0(C,\operatorname{SParAd}(E)\otimes L),
 \qquad\text{equivalently}\qquad
 \operatorname{Res}_p\phi\in\mathfrak n_p\quad(p\in D).
\end{equation}
\end{definition}
For a scheme $S$, let $\mathcal M_{\mathcal G}(S)$ be the
groupoid of triples $(E,\sigma,\phi)$, where $(E,\sigma)$
is an object of $\operatorname{Bun}_{\mathcal G}(S)$ and
\[
 \phi\in H^0(C\times S,\ad E\otimes\operatorname{pr}_C^*L),
 \qquad
 \operatorname{Res}_p\phi
 \in H^0(S,E_{\mathsf B,p}\times^{\mathsf B}\mathfrak n)
 \quad(p\in D).
\]
Here $\operatorname{pr}_C:C\times S\to C$ is the projection, and $\operatorname{Res}_p$ denotes the relative residue along $\{p\}\times S$. Morphisms are isomorphisms of $G$-bundles preserving the sections and the Higgs fields. For a morphism $f:S'\to S$, pullback is obtained by pulling back $E$ and $\phi$ along $\mathrm{id}_C\times f$ and each flag section $\sigma_p$ along $f$.

Choose a nondegenerate invariant symmetric form $\kappa$ on $\mathfrak g$. Since $\mathfrak b_p^\perp=\mathfrak n_p$,
\begin{equation}\label{eq:parabolic-cotangent}
 \operatorname{ParAd}(E)^\vee\otimes K_C
 \simeq\operatorname{SParAd}(E)\otimes L,
 \qquad
 \mathcal M_{\mathcal G}
 \simeq T^*\operatorname{Bun}_{\mathcal G},
\end{equation}
where $T^*\operatorname{Bun}_{\mathcal G}$ denotes the cotangent stack of $\operatorname{Bun}_{\mathcal G}$. Let $\mathfrak c_G=\mathfrak g/\!/G$. Choose homogeneous generators of $\C[\mathfrak g]^G$ of degrees $d_1,\ldots,d_{\rk G}$. The Hitchin base is
\begin{equation}\label{eq:hitchin-base}
B_G(C,D)=\bigoplus_{i=1}^{\rk G}H^0\bigl(C,L^{d_i}(-D)\bigr).
\end{equation}
For $D=\varnothing$, we call it \emph{the canonical Hitchin base}. Denote the chosen invariant generators by $f_i$. The Hitchin morphism is
\[
 h(E,\sigma,\phi)=(f_i(\phi))_i\in B_G(C,D),
 \qquad f_i(\phi)|_D=f_i(\operatorname{Res}_D\phi)=0.
\]
Thus its target is \eqref{eq:hitchin-base}. For $D=\varnothing$, the definitions give $\operatorname{ParAd}(E)=\operatorname{SParAd}(E)=\ad E$. For more details, see
\cite{Yun}.

Fix the class $\tau\in\pi_1(G)$ of the underlying topological
bundle of $E$. Choose a maximal torus $T\subset\mathsf B$, with cocharacter lattice $X_*(T):=\operatorname{Hom}(\mathbb G_m,T)$.
Let $\mathfrak A\subset X_*(T)\otimes_{\mathbb Z}\mathbb R$ be the open Weyl alcove whose closure contains $0$ and whose
associated parabolic subgroup is $\mathsf B$, using the conventions of \cite{BGM}.
Fix rational parabolic weights $\boldsymbol\alpha=(\alpha_p)_{p\in D}$ with
\[
 \alpha_p\in\mathfrak A\cap
 \bigl(X_*(T)\otimes_{\mathbb Z}\mathbb Q\bigr)
 \qquad(p\in D).
\]
For $D=\varnothing$ the weight tuple is empty. Parabolic degree and
(semi)stability are defined by \cite{BGM},
using holomorphic reductions of the flagged bundle above and the central parameter fixed by its parabolic degrees. We denote the $\boldsymbol\alpha$-semistable and stable loci of type $\tau$ by
\[
 \mathcal M_{\mathcal G}^{\tau,ss}(\boldsymbol\alpha),\qquad
 \mathcal M_{\mathcal G}^{\tau,s}(\boldsymbol\alpha)
 \subset\mathcal M_{\mathcal G}.
\]

We write $M_G(C,D)$ for the coarse moduli space of the semistable locus, with the fixed parameters $\tau$ and $\boldsymbol\alpha$ omitted from the notation.

\subsubsection{Deformations and symplectic forms}

For $x=(E,\sigma,\phi)$, put
\[
 \mathcal C_x^\bullet=
 \left[\operatorname{ParAd}(E)\xrightarrow{[\phi,-]}
             \operatorname{SParAd}(E)\otimes L\right],
 \qquad\deg\mathcal C_x^\bullet=0,1.
\]
The deformation calculation of \cite[Section~4.1]{BKV} gives
\begin{equation}\label{eq:parabolic-deformations}
\begin{gathered}
 \Lie\operatorname{Aut}(x)=\mathbb H^0(C,\mathcal C_x^\bullet),
 \qquad T_x\mathcal M_{\mathcal G}=\mathbb H^1(C,\mathcal C_x^\bullet),\\
 \mathbb H^i(C,\mathcal C_x^\bullet)
 \simeq\mathbb H^{2-i}(C,\mathcal C_x^\bullet)^\vee.
\end{gathered}
\end{equation}
The group $\mathbb H^2(C,\mathcal C_x^\bullet)$ is the obstruction space for extending a deformation of
$x=(E,\sigma,\phi)$ from an Artinian base to an infinitesimal thickening of that base, with $(C,D)$ fixed. Set
\[
 \mathcal C_{x,\mathrm{cen}}^\bullet
 :=
 [\mathfrak z_G\otimes\mathcal O_C
  \xrightarrow{0}\mathfrak z_G\otimes K_C].
\]
Let $\mathcal C_{x,\mathrm{der}}^\bullet
:=\ker(\mathcal C_x^\bullet\to
\mathcal C_{x,\mathrm{cen}}^\bullet)$,
where the map is induced by the projection onto
$\mathfrak z_G$ along $[\mathfrak g,\mathfrak g]$.
If $\Lie\operatorname{Aut}(x)=\mathfrak z_G$, then
\[
 \mathbb H^0(C,\mathcal C_{x,\mathrm{der}}^\bullet)
 =
 \mathbb H^2(C,\mathcal C_{x,\mathrm{der}}^\bullet)
 =0.
\]
Since $[\mathfrak z_G,\mathfrak g]=0$,
the above vanishing and
\cite[Propositions~5.47 and~5.49]{Manetti04}
imply that $\mathcal M_{\mathcal G}$ is smooth at $x$.

Let $Z_G^0$ be the identity component of the center of $G$, with Lie algebra $\mathfrak z_G$, and let
$r:\mathcal M_{\mathcal G}^{\tau,s}(\boldsymbol\alpha)\to\mathfrak X$ be its rigidification, in the sense of \cite[Appendix~A, Theorem~A.1]{AOV}:
\[
\operatorname{Aut}_{\mathfrak X}(r(x))=\operatorname{Aut}(x)/Z_G^0.
\]
In other words, $r:\mathcal M_{\mathcal G}^{\tau,s}(\boldsymbol\alpha)\to\mathfrak X$ is a gerbe banded by $Z_G^0$. Hence, $\mathfrak X$ is smooth. Since the automorphism groups have zero Lie algebra, $\mathfrak X$ is a Deligne--Mumford stack. 

Under \eqref{eq:parabolic-cotangent}, let $\lambda_\kappa$ be the Liouville one-form on $\mathcal M_{\mathcal G}$, and set $\omega_\kappa=d\lambda_\kappa$. The restriction of $\omega_\kappa$ to the stable locus descends to a closed nondegenerate two-form on $\mathfrak X$; see \cite[Sections~2 and~3.1--3.2]{BGL15} and \cite[Section~4.1, equation~(7)]{BKV}. We omit the subscript $\kappa$ when the form is fixed.

\subsection{Regular centralizers and Picard actions}

Let $\chi:\mathfrak g\to\mathfrak c_G$ be the adjoint quotient. Ng\^o constructs a smooth commutative regular-centralizer group scheme $J_G\to\mathfrak c_G$ and a homomorphism $\chi^*J_G\to I_G$ to the universal automorphism centralizer on $\mathfrak g$ \cite[Section~2.1]{Ngo}. Twisting by the Higgs coefficient line bundle and pulling back along a characteristic $b$ gives a group scheme $J_b\to C$. For Borel reductions, Yun constructs the compatible homomorphism to the flag-preserving centralizer and the corresponding Picard action \cite[Lemmas~2.3.1--2.3.3]{Yun}. Write $B=B_G(C,D)$ and let $J\to C\times B$ be the regular centralizer above. The relative Picard stack is, by definition,
\[
 \mathcal P(S)=\{J|_{C\times S}\text{-torsors on }C\times S\},
 \qquad S\longrightarrow B,
\]
with torsor isomorphisms as morphisms and contracted product as the group law. It is smooth over $B$
\cite[Definition~4.13 and Proposition~4.16]{DCFFM}.
For $x=(E,\phi,(E_{\mathsf B,p})_p)$ and $b=h(x)$, by \cite[Lemmas~2.3.1--2.3.3]{Yun}, the centralizer
homomorphisms $j_G,j_{\mathsf B}$ give
\begin{equation}\label{eq:parabolic-torsor-action}
 \mathcal T\cdot x
 =\left(\mathcal T\times^{J_b,j_G}(E,\phi),
       (\mathcal T|_p\times^{J_b|_p,j_{\mathsf B}}E_{\mathsf B,p})_p\right).
\end{equation}
Here $E_{\mathsf B,p}$ is the Borel torsor defined by $\sigma_p$,
$j_G:J_b\to\underline{\operatorname{Aut}}_C(E,\phi)$ is the
centralizer homomorphism, and $j_{\mathsf B}$ is its flag-preserving
factorization at $p$.
That action preserves $\mathfrak b_p\to\mathfrak b_p/\mathfrak n_p$; its zero fiber is
exactly the residue condition \eqref{eq:strongly-parabolic-field}.

Let $P=(\pi_0^{\mathrm{fppf}}\mathcal P)^0$ whenever this identity component is represented by a smooth group scheme; here
$\pi_0^{\mathrm{fppf}}\mathcal P$ is the fppf sheafification of the
presheaf of isomorphism classes. Set $\mathcal P^0=\mathcal P\times_{\pi_0^{\mathrm{fppf}}\mathcal P}P$
and write $\rho:\mathcal P^0\to P$ for the projection.
Under the finite-inertia and component-preservation hypotheses in Assumption~\ref{assum:parabolic-application}(1)--(2),
the torsor action \eqref{eq:parabolic-torsor-action} induces the compatible actions \begin{equation}\label{eq:parabolic-stack-action}
\begin{tikzcd}
 \mathcal P^0\times_B\mathfrak X\ar[r,"\widetilde c"]
       \ar[d,"\rho\times q"'] &\mathfrak X\ar[d,"q"]\\
 P\times_B X\ar[r,"c"']&X,
\end{tikzcd}
\end{equation}
where $q:\mathfrak X\to X$ is the coarse moduli morphism
\cite[Section~3]{AOV}.
The upper action is by $\mathcal P^0$. An action of $P$ itself on
$\mathfrak X$ would require further descent: for example,
\cite[Corollary~5.35]{DCFFM} rigidifies by the whole $Z_G$ and assumes
$\deg L>0$.

\subsubsection{The Hamiltonian identity}

Set $\mathcal W=\bigoplus_iL^{d_i}(-D)$, so $B=H^0(C,\mathcal W)$.
The regular-centralizer identity \cite[Proposition~4.12(iii)]{DCFFM}
and Serre duality give
\begin{equation}\label{eq:parabolic-cotangent-centralizer}
\begin{gathered}
 \Lie J_b\simeq\bigoplus_iL^{1-d_i}=\mathcal W^\vee\otimes K_C,\\
 \eta_b:H^1(C,\Lie J_b)\xrightarrow{\sim}H^0(C,\mathcal W)^\vee=T_b^*B.
\end{gathered}
\end{equation}
Differentiating $j_G,j_{\mathsf B}$ and the invariant polynomials gives
maps of complexes
\[
\Lie J_b[0]\xrightarrow{(dj_\phi,0)}\mathcal C_x^\bullet
\xrightarrow{(0,d\chi_\phi)}\mathcal W[-1].
\]
The two maps of complexes are dual under the pairing induced by $\kappa$, as expressed by the local identity
\[
\kappa(dj_\phi(s),v)=\langle s,d\chi_\phi(v)\rangle
 \quad(s\in\Lie J_b,\ v\in\operatorname{SParAd}(E)\otimes L);
\]
see \cite[Lemma~A.9 and Remark~A.11]{DCHMSupport} and \cite[Lemma~7.11]{DCFFM}. For the corresponding global duality in the unmarked
$K_C$-valued case, see \cite[Proposition~A.12]{DCHMSupport}.

For $\xi\in H^1(C,\Lie J_b)$, let $\xi^\#\in T_x\mathfrak X$ be the tangent vector obtained by acting on $x$ with the first-order deformation of the trivial $J_b$-torsor represented by $\xi$.
Taking first hypercohomology of the two maps above gives
\[
 H^1(C,\Lie J_b)
 \xrightarrow{\xi\mapsto\xi^\#}
 T_x\mathfrak X
 \xrightarrow{dh_x}
 T_bB=H^0(C,\mathcal W),
\]
where $dh_x$ is the differential of the Hitchin map. Hence
\begin{equation}\label{eq:parabolic-hamiltonian}
\omega_x(\xi^\#,w)=-\langle\eta_b(\xi),dh_x(w)\rangle, \qquad \xi\in H^1(C,\Lie J_b),\quad w\in T_x\mathfrak X.
\end{equation}
Let $p:C\times B\to B$ be the projection. When the differential of
$\rho$ induces an isomorphism
$\theta:R^1p_*(\Lie J)\xrightarrow{\sim}\Lie(P/B)$,
\eqref{eq:parabolic-cotangent-centralizer} gives
$\eta:\Lie(P/B)\xrightarrow{\sim}\Omega_B^1$ via $\theta^{-1}$.
Then \eqref{eq:parabolic-hamiltonian} reads
$\iota_{\xi^\#}\omega=-h^*\eta(\xi)$ for local sections
$\xi$ of $\Lie(P/B)$.

\subsection{Geometric hypotheses for the applications}
\label{subsec:parabolic-application-hypotheses}

The following conditions are used in multiplicativity and in Theorem~\ref{thm:intro-weak}.
They concern the complete semistable moduli space in the chosen topological type and with the specified weights.
\begin{assumption}\label{assum:parabolic-application}
Let $X$ be a chosen Higgs moduli space with Hitchin base $B=B_G(C,D)$, and put $d=\dim B$. Assume the following conditions are satisfied:
\begin{enumerate}
\item The chosen component satisfies
\[
 \varnothing\ne\mathcal M_{\mathcal G}^{\tau,ss}(\boldsymbol\alpha)
 =\mathcal M_{\mathcal G}^{\tau,s}(\boldsymbol\alpha).
\]
The rigidified stack $\mathfrak X$ above has finite inertia, meaning
that
\[
 I_{\mathfrak X}:=
 \mathfrak X\times_{\mathfrak X\times\mathfrak X}\mathfrak X
 \longrightarrow\mathfrak X
 \quad\text{is finite}.
\]
Its coarse moduli space is $X$.
\item The action of $\mathcal P^0$ by twisting with $J_b$-torsors preserves the chosen stable component $\mathcal M_{\mathcal G}^{\tau,s}(\boldsymbol\alpha)$, \eqref{eq:parabolic-stack-action} holds, and the natural map induced by $\rho$ is an isomorphism:
\[
 \theta:R^1p_*(\Lie J)\xrightarrow{\sim}\Lie(P/B).
\]
Moreover $(X,P,B)$ is a weak abelian fibration in the sense of
Definition~\ref{def:weak-abelian}.
\item The Hitchin map $h_X:X\to B$ is projective and surjective, and
\[
 \dim X=2d,\qquad
 X_b\text{ is connected and pure of dimension }d\quad(b\in B).
\]
Moreover,
\[
 \dim\{b\in B:\dim\Ab(P_b)\le c\}\le c\qquad(c\ge0).
\]
\end{enumerate}
\end{assumption}

Under Assumption~\ref{assum:parabolic-application}, the preceding deformation calculation shows that $\mathfrak X$ is smooth. Its finite inertia implies that $X$ has finite quotient singularities, hence is a rational homology manifold.

\begin{assumption}\label{assum:picard-equivariance}
For a comparison $(F,a):(X\to B_X)\to(Y\to B_Y)$, there is a homomorphism
$\varphi:P_X\to a^*P_Y$ over $B_X$ such that
\[
 F_a(p\cdot x)=\varphi(p)\cdot F_a(x),
 \qquad F_a(x)=(F(x),h_X(x)).
\]
\end{assumption}

Under this assumption, we take $Q=P_X$, $u=\id$ and $v=\varphi$ in Theorem~\ref{thm:common-action-weak}. The effective-rank inequality is then the rank inequality in Assumption~\ref{assum:parabolic-application}(3). For the diagonal, equivariance holds with $\varphi=\Delta_P$, as used in Theorem~\ref{thm:multiplicativity} below.

\subsection{Proof of multiplicativity}
In this subsection, we prove the multiplicativity for the perverse truncation of direct image sheaves.

The product $h\times h:X\times X\to B\times B$ is the
Hitchin fibration for $G\times G$, with the prescribed
moduli data on each factor. Indeed,
\[
 \mathbb C[\mathfrak g\oplus\mathfrak g]^{G\times G}
 =\mathbb C[\mathfrak g]^G\otimes\mathbb C[\mathfrak g]^G;
\]
see the construction in
\cite[Definitions~3.7--3.8]{DCFFM}. The diagonal homomorphism $G\to G\times G$ induces $\Delta_X$ on these moduli spaces and $\Delta_B$ on
their Hitchin bases, with
\[
 (h\times h)\circ\Delta_X=\Delta_B\circ h.
\]

\begin{theorem}\label{thm:multiplicativity}
Let $h:X\to B$ satisfy Assumption~\ref{assum:parabolic-application}, and put $K=Rh_*\mathbb Q_X$ and $d=\dim B$. Then, for $r,s\in\mathbb Z$, the cup product factors as
\begin{equation}\label{eq:multiplicative-truncations}
 \ptau_{\le r}K\otimes^{\mathbf L}\ptau_{\le s}K
 \longrightarrow\ptau_{\le r+s-d}K\longrightarrow K.
\end{equation}
In particular $P_iH^m(X,\Q)\cup P_jH^n(X,\Q)\subseteq P_{i+j}H^{m+n}(X,\Q)$.
\end{theorem}
\begin{proof}
Consider the diagonal maps
\[
 F=\Delta_X:X\to X\times X,\qquad
 a=\Delta_B:B\to B\times B,
\]
which satisfy $(h\times h)\circ F=a\circ h$.
Take $Q=P$, with homomorphisms
\[
 P\xleftarrow{\mathrm{id}}Q=P
 \xrightarrow{\Delta_P}\Delta_B^*(P\times P).
\]
Both homomorphisms have trivial kernel, so the effective rank at $b\in B$ is $\dim\operatorname{Ab}(P_b)$. The map $h\times h$ is projective, its total space has dimension $4d$, and its fibers are connected and pure of dimension $2d$. Moreover, $\operatorname{IC}_{X\times X}\simeq\operatorname{IC}_X\boxtimes\operatorname{IC}_X$, so $X\times X$ is a rational homology manifold. Since
\[
 {}^p\tau_{\le r}K\boxtimes{}^p\tau_{\le s}K
 \in {}^pD^{\le r+s}(B\times B),
\]
Theorem~\ref{thm:common-action-weak}, with $k=r+s-2d$, gives
\[
\begin{aligned}
 \ptau_{\le r}K\otimes^{\mathbf L}\ptau_{\le s}K
 &\simeq\Delta_B^*\bigl(\ptau_{\le r}K\boxtimes\ptau_{\le s}K\bigr)\\
 &\longrightarrow\Delta_B^*\ptau_{\le r+s}(K\boxtimes K)
 \longrightarrow\ptau_{\le r+s-d}K\longrightarrow K.
\end{aligned}
\]
Under proper K\"unneth, pullback along $\Delta_X$ identifies this composite with the cup product on $K$. Taking $r=d+i$, $s=d+j$ and hypercohomology proves the assertion for $P_\bullet$ by
\eqref{eq:cohom-P}.
\end{proof}

\begin{remark}
Iterating \eqref{eq:multiplicative-truncations} gives, for every integer $\ell\ge2$ and $r_1,\ldots,r_\ell\in\Z$, factorizations
\[
 \bigotimes_{i=1}^{\ell}\ptau_{\le r_i}K
 \longrightarrow\ptau_{\le \sum_{i=1}^{\ell}r_i-(\ell-1)d}K
 \longrightarrow K
\]
whose composite is the $\ell$-fold cup-product.
\end{remark}

Replace the rank inequality in Assumption~\ref{assum:parabolic-application}(3) by $\dim\{b\in B:\dim\Ab(P_b)\le c\}\le c+\kappa$ for all $c\ge0$, with $\kappa\in\mathbb Z_{\ge0}$, and keep the other hypotheses.
\begin{corollary}\label{cor:multiplicativity-rank-defect}
For every $r,s\in\mathbb Z$, the cup product on $K=Rh_*\mathbb Q_X$ factors as
\[
 \ptau_{\le r}K\otimes^{\mathbf L}\ptau_{\le s}K
 \longrightarrow\ptau_{\le r+s-d+\kappa}K\longrightarrow K.
\]
Consequently $P_iH^m(X,\Q)\cup P_jH^n(X,\Q)
\subseteq P_{i+j+\kappa}H^{m+n}(X,\Q)$.
\end{corollary}
\begin{proof}
Repeat the arguments in Theorem~\ref{thm:multiplicativity} with Corollary~\ref{cor:common-action-rank-defect} at $k=r+s-2d$.
\end{proof}

Let $C$ be a smooth projective curve of genus at least two,
and fix $N\ge2$, $e\in\Z$ and $\Lambda\in\Pic^e(C)$ with
$\gcd(N,e)=1$.
Let $M_{N,e}$ be the moduli space of semistable rank-$N$,
degree-$e$, $K_C$-valued Higgs bundles, and let
$M_{\SL_N}^{\Lambda,0}(C,\varnothing)$ denote its closed subspace
defined by $\det E\simeq\Lambda$ and $\operatorname{tr}\phi=0$.

\begin{corollary}\label{cor:MSY}
Let $X=M_{N,e}$ or $X=M_{\SL_N}^{\Lambda,0}(C,\varnothing)$,
and let $h:X\to B$ be its Hitchin morphism.
Put $K=Rh_*\Q_X$ and $d=\dim B$.
Then cup-product satisfies
\eqref{eq:multiplicative-truncations} for all $r,s\in\Z$.
In particular,
\[
 P_iH^a(X,\Q)\cup P_jH^b(X,\Q)
 \subseteq P_{i+j}H^{a+b}(X,\Q).
\]
\end{corollary}
\begin{proof}
The geometric hypotheses follow from
\cite[Section~2.2]{DCHMSupport} and
\cite[Section~1.1]{DMSSLPW}.
For the fixed-determinant space, the spectral norm identifies
the acting neutral Prym group with the $\SL_N$ regular-centralizer
Picard group; preservation of stability and affine stabilizers
are inherited from the $\GL_N$ action.
The required group properties, Lie identification and full-base
$\delta$-regularity follow from
\cite[Propositions~4.12, 4.17, 4.29, Corollary~4.27 and Theorem~7.16]{DCFFM}.
The rational Betti polarization is supplied by
\cite[Theorem~1.2]{AFPolarizability}.
Apply Theorem~\ref{thm:multiplicativity}.
\end{proof}

In particular, the cohomological consequence gives another proof
of multiplicativity of the perverse filtration for coprime
$\GL_N$ Hitchin moduli
\cite{MSPW,HMMSSPW,MSYFourier}
and fixed-determinant traceless $\SL_N$ Hitchin moduli
\cite[Theorem~1.1]{WenSLMult}.
For related motivic and sheaf-theoretic multiplicativity results
of Maulik--Shen--Yin, see
\cite[Theorem~2.6(ii)--(iii)]{MSYFourier},
\cite{MSYCycles}, and
\cite[Theorem~3.2(b)]{MSYDualizable}.

\providecommand{\bysame}{\leavevmode\hbox to3em{\hrulefill}\thinspace}

\end{document}